\documentclass[
final
]{dmtcs-episciences}

\usepackage[utf8]{inputenc}

\usepackage[english]{babel}
\usepackage{amsthm,amsmath,amssymb}
\usepackage[small,bf,hang]{caption} 
\usepackage{longtable}
\usepackage{subfig} 

\usepackage{multirow}
\usepackage{array}

\usepackage{algorithm}
\usepackage{algpseudocode}

\newcolumntype{L}[1]{>{\raggedright\arraybackslash}p{#1}}
\newcolumntype{C}[1]{>{\centering\arraybackslash}p{#1}}
\newcolumntype{R}[1]{>{\raggedleft\arraybackslash}p{#1}}

\newtheorem{theorem}{Theorem}
\theoremstyle{plain}
\newtheorem{lemma}[theorem]{Lemma}

\theoremstyle{definition}
\newtheorem{question}[]{Question}

\newtheorem*{procedure*}{Procedure}

\graphicspath{{figures/}}

\author{Geoffrey Exoo\affiliationmark{1}\thanks{Supported by Grant No.\ 1751765 from the National Science Foundation, USA.}
  \and Jan Goedgebeur\affiliationmark{2,3}\thanks{Supported by a Postdoctoral Fellowship of the Research Foundation Flanders (FWO).}}
\title{Bounds for the smallest $k$-chromatic graphs of given girth}
\affiliation{
  Department of Mathematics and Computer Science, Indiana State University, USA\\
  Department of Applied Mathematics, Computer Science and Statistics, Ghent University, Belgium\\
  Computer Science Department, University of Mons, Belgium}
\received{2018-6-11}
\revised{2019-2-14}
\accepted{2019-2-15}

\begin{document}
\publicationdetails{21}{2019}{3}{9}{4576}
\maketitle
\begin{abstract}
Let $n_g(k)$ denote the smallest order of a $k$-chromatic graph of girth at least~$g$.
We consider the problem of determining $n_g(k)$ for small values of $k$ and $g$.
After giving an overview of what is known about $n_g(k)$, we
provide some new lower bounds based on exhaustive searches, and
then obtain several new upper bounds using computer
algorithms for the construction of witnesses, and for the verification
of their correctness. 
We also present the first examples of reasonably small order for $k = 4$ and $g > 5$.
In particular, the new bounds include:
$n_4(7) \leq 77$, $26 \leq n_6(4) \leq 66$ and $30 \leq n_7(4) \leq 171$.
\end{abstract}


\section{Introduction}
\label{section:intro}

This paper deals with the problem of determining the minimum order among
graphs with given girth $g$ and chromatic number $k$. The \textit{chromatic number} of a graph is the minimum number of colours required to colour the vertices of the graph such that no two adjacent vertices have the same colour. The \textit{girth} of a graph is the length of its shortest cycle.

In a well-known demonstration of the power of the probabilistic
method Erd\H{o}s \cite{Erd59} established in 1959 the existence of graphs
for which both the girth and the chromatic number are arbitrarily large.
This result followed earlier efforts from the early fifties of Zykov \cite{zykov1949some}, Descartes \cite{descartes1954}, and Kelly and Kelly \cite{kelly1954paths}
who constructed graphs for girth less than or equal to six and with
arbitrarily large chromatic numbers.
At around the same time
an important construction was discovered by Mycielski \cite{M55}, who
showed how to use a $k$-chromatic triangle free graph
of order $n$ to construct a $(k+1)$-chromatic triangle free graph of order $2n+1$.
Others, including  Lov{\'a}sz~\cite{lovasz1968chromatic}, Kostochka and Ne\v{s}et\v{r}il~\cite{kostnesi}, and Alon et al.~\cite{alon2016coloring}, have provided constructions of graphs with given chromatic number and girth.

Because the actual graphs produced by these methods are extremely large, especially for
girth five and up,
there have been efforts to identify the smallest  graphs for each value of $g$ and $k$.
To this end, let $n_g(k)$ denote the order of the smallest $k$-chromatic graph with girth at
least $g$.  

Chv{\'a}tal~\cite{C74} showed in 1974 that the Gr\"otzsch graph is the smallest
triangle-free $4$-chromatic graph, so $n_4(4)=11$.
In~\cite{toft1988} Toft asked
for the value of $n_4(5)$.
The Mycielski construction immediately
gives an upper bound $n_4(5) \leq 23$.
Using a computer search, Grinstead, Katinsky and Van Stone~\cite{grinstead1989minimal}
showed that $21 \leq n_4(5) \leq 22$.  The issue was settled in 1995 by
Jensen and Royle~\cite{jensen1995small} who established the exact value $n_4(5)=22$.
Note that $n_4(k)$ is equal to the value of the vertex Folkman number $F_v(2^{k-1};3)$~\cite{xu2016chromatic}.

In a posting on {\it StackExchange} from 2015,
Droogendijk~\cite{droogendijk2015} showed that $n_4(6) \leq 44$, improving the
upper bound of $45$ derived from the Mycielski construction.
In the arXiv manuscript~\cite{goedgebeur2017} the second author recently lowered this bound to $40$, and also
established the bounds
$32 \leq n_4(6)$, $41 \leq n_4(7) \leq 81$,
$29 \leq n_5(5)$ and $25 \leq n_6(4)$. 

In~\cite[Section 7.3]{jensen1995graph} Jensen and Toft asked for the value of $n_5(4)$. The
Brinkmann graph~\cite{brinkmann1997smallest} gives an upper bound of $n_5(4) \leq 21$.
Royle~\cite{royle2015} showed that $n_5(4) = 21$ and that there are
exactly 18 $4$-chromatic graphs of girth at least 5 on 21 vertices.

Asymptotic bounds on $n_4(k)$ are discussed in~\cite[Section 7.3]{jensen1995graph}.
The bounds are closely related to results on the classical Ramsey numbers $R(3,t)$.
It is shown that there exist positive constants $c_1$ and $c_2$ such that
\[
c_1 k^2 \log k \leq n_4(k) \leq c_2 (k \, \log k)^2.
\]

Kim's~\cite{kim1995ramsey} lower bound on $R(3,t)$ which established that $R(3,t) = \Theta(t^2 / \log t)$  implies that $n_4(k) = \Theta(k^2 \log k)$.

For larger girth, the best known asymptotic lower bound appears to be based on
the well-known Moore bound on the order of graphs with given minimum degree
and girth~\cite{exoo2008dynamic}.
Recall that a \textit{$k$-vertex-critical} graph is a $k$-chromatic graph such that every
proper induced subgraph is $(k-1)$-colourable.  Such a graph has minimum degree at
least $k-1$.
Using minimum degree $k-1$ and the Moore bound, we obtain the following bound for
odd girth $g$:

\[
n_g(k) \geq \frac{(k-1)(k-2)^{(g-1)/2} - 2}{k-3}.
\]

Similarly for even girth we have:

\[
n_g(k) \geq \frac{2(k-2)^{g/2} - 2}{k-3}.
\]

In this paper we obtain
new computational lower and upper bounds for
$g \leq 7$ and $k \leq 7$, and
describe the construction methods used for the upper bounds.

In Table~\ref{table:bounds} we give an overview of (to the best of our
knowledge) the current bounds for $n_g(k)$. 
The known exact values of $n_g(k)$ are listed as vertically centred values and the lower and upper bounds appear as top and bottom entries, respectively.

\begin{table}[htb!]
\centering
 \renewcommand{\arraystretch}{1.25} 
    \begin{tabular}{|lr||C{40pt}|C{40pt}|C{40pt}|C{40pt}|}
      \hline   
  & $g$  &  4   & 5    & 6    & 7\\  
$k$ & &     &     &     & \\    
  \hline\hline        
      \multirow{2}{*}{4} &  & \multirow{2}{*}{11}   & \multirow{2}{*}{21}   & \textbf{26}  & \textbf{30} \\
        & &    &    & \textbf{66}  & \textbf{171}\\    
        \hline  
      \multirow{2}{*}{5}  & & \multirow{2}{*}{22}   & 29   & \textbf{33}  & \textbf{66} \\
        & &    & 80   &   & \\  
        \hline      
      \multirow{2}{*}{6} &  & 32   & \textbf{36}   & \textbf{51}  & \textbf{127} \\
        & & 40   &    &   & \\      
        \hline    
      \multirow{2}{*}{7} &  & 41   & \textbf{45}   & \textbf{73}  & \textbf{218} \\
        & & \textbf{77}   &    &   & \\  
        \hline      
      \multirow{2}{*}{8} &  & 51   & \textbf{57}   & \textbf{99}  & \textbf{345} \\
        & & \textbf{155}   &    &   & \\              
      \hline      
    \end{tabular}
\caption{Known nontrivial values and bounds for $n_g(k)$. The new bounds
determined in this paper are marked in bold. In Section~\ref{sect:LB} we describe how we obtained the new lower bounds and in Section~\ref{sect:UB} how we obtained the new upper bounds.}
\label{table:bounds}
\end{table}

The precise determination of the chromatic number for several of our graphs required extensive
computations.  While the chromatic number claims for some of the smaller graphs can be quickly verified
using packages like Sage, Maple or Mathematica, others required hours of computation
spread across hundreds of multicore CPUs.
For each of the graphs which yield a new upper bound in
Table~\ref{table:bounds}, the chromatic number has been verified by two
independent algorithms (one implemented by each author) and all results
were in complete agreement.

In the next section, we give details on our improvements on the lower
bounds.
Then in Section~\ref{sect:UB} we discuss the methods used to obtain the
new upper bounds. Finally, in Section~\ref{sect:open_problems} we conclude with some open problems.


\section{Improving lower bounds for $n_g(k)$}
\label{sect:LB}

In this section we review some useful facts about $k$-chromatic graphs and then
present two formulas that give general lower bounds for
$n_g(k)$.  We then provide more detailed computational arguments for the cases
$n_6(4)$ and $n_7(4)$.

\begin{lemma} The following general lower bound for $n_g(k)$ holds for $g \geq 4$:
\begin{equation*}\label{eq:lb1}
n_g(k) \geq n_g(k-1) + max(k, \lceil 3(k-2)/2 \rceil) + 1
\end{equation*}
\end{lemma}

\begin{proof}
It follows from Brooks'
Theorem~\cite{brooks1941colouring} that a connected $k$-chromatic graph which is not a  complete graph or an odd cycle must have maximum degree at least $k$.
Kostochka~\cite{kostochka1982modification} proved that triangle-free graphs have chromatic number at most $\frac{2\Delta}{3}+2$, so the maximum degree of a $k$-chromatic triangle-free graph is at least $\lceil 3(k-2)/2 \rceil$.

Note that removing a vertex of degree $d$ and its $d$ neighbours from a
$k$-vertex-critical graph $G$ of girth $g \geq 4$ yields a $(k-1)$-chromatic
graph of girth at least $g$ on $|V(G)| - d - 1$ vertices.
This observation gives us the general lower bound from the statement.
\end{proof}


The second general condition is obtained by a variation of the argument used to
establish the Moore bound, which is obtained by counting the number of
vertices which are at distance at
most $\lfloor (g-1)/2 \rfloor$ from a central vertex for odd $g$ or a central
edge for even $g$.
We note again that
the Moore bound for the order of a smallest  graph of minimum degree $d$ and girth $g$ is:
\begin{lemma}[Moore bound]
\begin{equation*}
  \begin{cases}
   \  \frac{d(d-1)^{(g-1)/2}-2}{d-2} & \quad \text{if } g \text{ is odd}\\\\
   \  \frac{2(d-1)^{g/2} - 2}{d-2}  & \quad \text{if } g \text{ is even}
  \end{cases}
\end{equation*}
\end{lemma}

This can be used to prove the following improved lower bounds for $n_g(k)$:

\begin{lemma}\label{eq:lb2} The following general lower bounds for $n_g(k)$ hold:
\begin{align} 
\begin{split}
n_4(k) &\geq (k-1) + k + k - 2 = 3k - 3 \nonumber \\
n_5(k) &\geq (k-1)k+1 = k^2 - k + 1 \nonumber \\
n_6(k) &\geq 2(k-2)(k-1)+2+k-1 + k - 2 = 2k^2 - 4k + 3 \nonumber \\
n_7(k) &\geq ((k-1)(k-2)+1)k + 1 = k^3 - 3k^2 + 3k + 1 \nonumber
\end{split}
\end{align}
\end{lemma}

\begin{proof}
A $k$-vertex-critical graph of girth $g$ has minimum degree at least
$k-1$ and maximum degree at least $k$, so we modify the Moore bound
argument by using a degree $k$ vertex in the central position. 
The idea is illustrated in
Figures~\ref{fig:moore_bound_girth5} and~\ref{fig:moore_bound_girth6} for
the case $k=4$ and $g=5$, and for the case $k=4$ and $g=6$. This yields the formulas from the theorem for the odd girth case.

\begin{figure}[h!tb]
	\centering
    \subfloat[]{\label{fig:moore_bound_girth5}\includegraphics[width=0.3\textwidth]{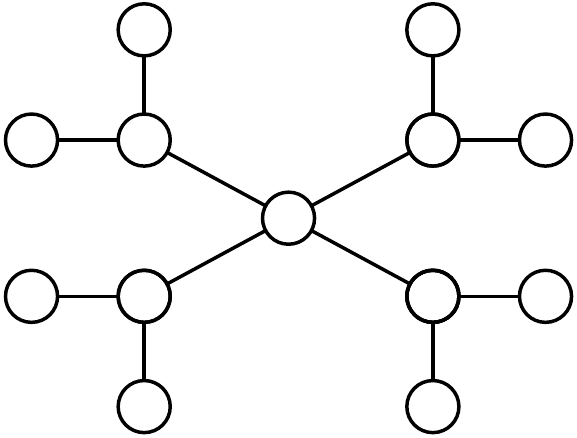}}	\qquad \qquad
    \subfloat[]{\label{fig:moore_bound_girth6}\includegraphics[width=0.3\textwidth]{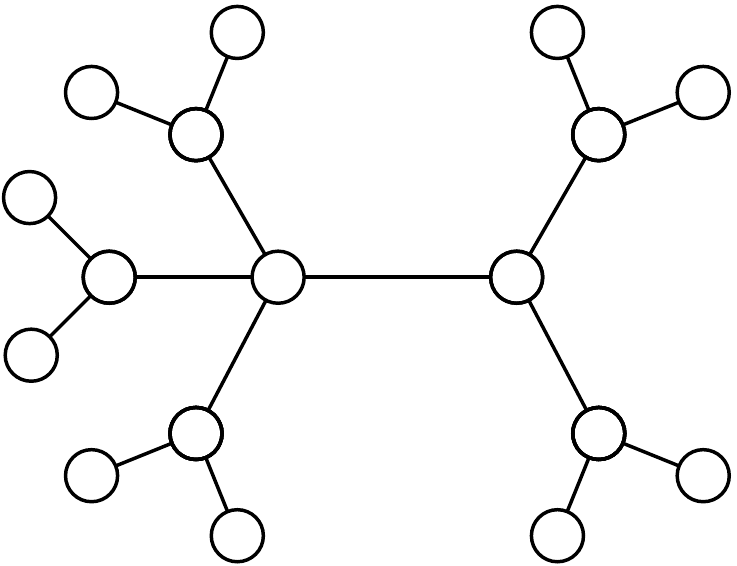}}   
	\caption{Construction for the minimum possible order of graphs with minimum degree 3 and maximum degree 4 with girth 5 and 6, respectively.}
	\label{fig:moore_bound}
\end{figure}

For the even girth case, we can say a little more. Here the extremal graphs are bipartite unless there are vertices at distance $g/2$ from the base edge. Adding a vertex can increase the chromatic number of a graph by at most one, so in a $k$-chromatic graph of girth $g$ there must be at least $k-2$ vertices at distance at least $g/2$ from the base edge. 
\end{proof}


The algorithm used in~\cite{goedgebeur2017} exhaustively generates all
triangle-free $k$-chromatic graphs from a given order by starting from the properly chosen set of triangle-free $(k-1)$-chromatic  graphs and adding a new vertex with a
given number of neighbours and connecting the neighbours to independent sets of
the source graphs in all possible ways. This algorithm can also be adapted to
generate all $k$-chromatic graphs of higher girth (and this was indeed used
in~\cite{goedgebeur2017} to show that  $n_5(5) \geq 29$). However, this method
is not effective to generate $k$-chromatic graphs of girth at least 6, since
the number of $(k-1)$-chromatic source graphs that the algorithm would have to
handle is huge. 

However we did computationally obtain the following new lower bounds using an alternative method.


\begin{theorem}
$n_6(4) \geq 26$ and $n_7(4) \geq 30$.
\end{theorem}

\begin{proof}
It follows from Lemma~\ref{eq:lb2} that $n_6(4) \geq 19$ and $n_7(4) \geq 29$.

We extended the generator \verb|geng|~\cite{nauty-website, mckay_14}
to generate graphs with girth at least 6 and girth at least 7 efficiently.
The source code of our plugin for \verb|geng| can be found in the ``ancillary material" of this paper.

By using the same reasoning as in the proof of Lemma~\ref{eq:lb2}, it is easy to see that a $4$-vertex-critical graph of girth at least 6 with maximum degree at least 7 must have at least 26 vertices and that a 4-vertex-critical graph of girth at least 7 with maximum degree at least 5 must have at least 36 vertices.

Using our extended version of \verb|geng|, we generated all graphs with minimum degree at least 3, maximum degree at most 6 and girth at least 6 from 19 up to 25 vertices and all graphs with minimum degree at least 3, maximum degree 4 and girth at least 7 on 29 vertices. We verified that all of these graphs are $3$-colourable, which yields the improved lower bounds from the statement. These computations were executed on a cluster and required roughly 2.5 and 12 CPU years for girth 6 and 7, respectively.
\end{proof}

The lower bounds listed in Table~\ref{table:bounds} for $n_4(k)$ with $k \geq
7$, for $n_5(k)$ with $k \geq 6$, for $n_6(k)$ with $k \geq 5$ and for $n_7(k)$
with $k \geq 5$ were obtained by taking the maximum of the formulas in Lemmas~\ref{eq:lb1} and~\ref{eq:lb2} as it was infeasible to improve these theoretical bounds
using our computational methods.


\section{Improving upper bounds for  $n_g(k)$}
\label{sect:UB}

In this section we present constructions for $k$-chromatic graphs of given girth and the graphs we obtained with it which establish an improved upper bound for $n_g(k)$. The adjacency lists of these graphs (as well as their \textit{graph6} encoding) can also be found in the ``ancillary material" of this paper on arXiv.

\subsection{Constructions for triangle-free $k$-chromatic graphs}
\label{subsect:UB_tf}

The construction by Mycielski~\cite{M55} is a classical construction for triangle-free
graphs of arbitrarily large chromatic number.
It yields an upper bound of $n_4(k+1) \leq 2n_4(k) + 1$.
In an interesting web posting Droogendijk~\cite{droogendijk2015} proposed the construction given below. This is a generalisation of a construction used by Jensen and Royle in Lemma~3 of~\cite{jensen1995small}.
In our outline of the procedure
we make extensive use of the following notation.
Given a graph $G$ and
a vertex $w \in V(G)$, we denote the set of neighbours
of $w$ by $N(w,G)$ or, if $G$ is clear from context, simply $N(w)$.

\medskip

\begin{procedure*}[Droogendijk~\cite{droogendijk2015}]\label{thm:droogendijk}
\textit{Let $G$ be a triangle-free $k$-chromatic graph on $n$ vertices and $S$ an
independent set such that no $(k-2)$-colouring of the
non-neighbours of $S$ can be extended to a $(k-1)$-colouring of $G-S$. Then the
triangle-free graph $G^*$ on $2n + 2 - |S|$ vertices which is constructed as described below
is $(k+1)$-chromatic\footnote{Note: As will be seen later not every graph
constructed in this way will be $(k+1)$-chromatic.}.} 
\end{procedure*}

Let $A$ be the set of neighbours of $S$, that is, $A = \{v \ | \ v \in N(w) : w
\in S\}$ and let $B$
be the set of non-neighbours of $S$, that is: $B = V(G) \setminus (S \cup
A)$.
The graph $G^*$ will have vertex set $V(G) \cup A' \cup B' \cup \{\alpha,\beta\}$.
$A'$ is an additional set of vertices $|A'|=|A|$.  Fix a one-to-one correspondence
between $A$ and $A'$.
Similarly,
$B'$ is an additional set of vertices $|B'|=|B|$.  Fix a one-to-one correspondence
between $B$ and $B'$.
Add edges between each vertex of $A'$ and the neighbours of the corresponding
vertex of $A$.
Similarly,
add edges between each vertex of $B'$ and the neighbours of the corresponding
vertex in $B$.
Finally, add two special vertices
$\alpha$ and $\beta$ which are adjoined to all vertices in $S \cup B'$
and $A' \cup B'$, respectively.
Note that if $G$ is $k$-chromatic and $|S| = 1$, $G[B]$ cannot be
$(k-2)$-colourable so in that case the conditions of the above procedure are
always fulfilled.
This construction will frequently produce $(k+1)$-chromatic graphs which
are smaller than those obtained by the Mycielski construction (i.e., when $|S| > 1$).

There are situations where $G^*$ is not $(k+1)$-chromatic.  For example,
let $G$ be a $9$-cycle (hence $3$-chromatic) with vertices $v_0, \dots, v_8$,
labelled cyclically.
Also let $S = \{v_0,v_3\}$, so $A = \{v_1,v_2,v_4,v_8\}$, and $B = \{v_5,v_6,v_7\}$.
Then $G^*$ turns out to be a $3$-chromatic graph on $18$ vertices. There are
several other ``counterexamples" for larger values of $k$.

Nevertheless, the construction method is very effective at obtaining
triangle-free $(k+1)$-chromatic graphs and yielded the following improved upper
bound for $n_4(7)$.


\begin{theorem}\label{thm:ub_n47}
$n_4(7) \leq 77$.
\end{theorem}
\begin{proof}
We implemented a computer program which searches for independent sets $S$ with
the required properties from Droogendijk's procedure in the input graphs and
which applies the construction to them.
We executed this program on the more than 750\,000 triangle-free
$6$-chromatic graphs on $40$ vertices from~\cite{goedgebeur2017}.
This yielded
several triangle-free $7$-chromatic graphs on $77$ vertices and no smaller ones.
Our specialised programs required approximately 100 hours per graph to verify that
these graphs are indeed $7$-chromatic.
\end{proof}

We also tried the
method of recursively adding and removing edges (as long as the graphs stay
$7$-chromatic and triangle-free) from~\cite{goedgebeur2017} on the $7$-chromatic graphs of order 77 from Theorem~\ref{thm:ub_n47}. This yielded several additional
$7$-chromatic graphs, but all of them were $7$-vertex-critical. The adjacency
list of the most symmetric triangle-free $7$-chromatic graph on $77$ vertices we found (i.e., a graph with an automorphism group of size 10) is listed in the Appendix. This graph can also
be downloaded from the database of interesting graphs from the \textit{House of
Graphs}~\cite{hog} by searching for the keywords ``triangle-free $7$-chromatic''\footnote{This graph can also be accessed directly at \url{https://hog.grinvin.org/ViewGraphInfo.action?id=30631}}. 

The bound $n_4(8) \leq 155$ from Table~\ref{table:bounds} is obtained by applying the Mycielski construction to one of our triangle-free $7$-chromatic graphs of order 77. (We could not apply Droogendijk's procedure here to obtain a better bound since the graphs are too big to perform the chromatic number computations in reasonable time).

The 750\,000 triangle-free $6$-chromatic graphs on $40$ vertices
from~\cite{goedgebeur2017} all have an automorphism group of size 1 or 2. Using
the LCF method (see Section~\ref{subsect:UB_higher_girth}) we were able to
obtain a triangle-free $6$-chromatic graph on $40$ vertices with an
automorphism group of size 10. It can be found in the Appendix or inspected at
the \textit{House of
Graphs}~\cite{hog} by searching for the keywords ``triangle-free $6$-chromatic * groupsize 10''\footnote{This graph can also be accessed directly at \url{https://hog.grinvin.org/ViewGraphInfo.action?id=30633}}.

\subsection{Constructions for $k$-chromatic graphs of girth at least 5}
\label{subsect:UB_higher_girth}

For girth larger than four, much less is known.  The only specific
value of $n_g(k)$ is $n_5(4) = 21$, due to the Brinkmann graph~\cite{brinkmann1997smallest}
which established $n_5(4) \leq 21$,
and Royle~\cite{royle2015} who established the exact value.



Attempting to search for graphs with girth $g > 4$ and chromatic number $k > 3$ requires
considering larger graphs.  
It was evident that any such example graph would be so large
that it would not be feasible
to check all graphs of the relevant orders.
So we considered some smaller search spaces, as has been
done for some related problems.
For example, the
early results on Ramsey numbers~\cite{kalb} were obtained by
limiting searches to \textit{circulant}
graphs, i.e., graphs admitting a cyclic automorphism of degree $n$.
Other searches, including those for cages
and for the degree-diameter problem,
focused on Cayley graphs and voltage graphs \cite{exoo2008dynamic,ddmdynamic}.

Following suit we began by looking at Cayley graphs.
For $4$-chromatic graphs of girth~$6$, the smallest Cayley graph we found
has order 96.  This $5$-regular graph is generated by the following
three permutations of degree $12$.  

\vspace{2mm}
{\tt
\begin{tabular}{l}
(1, 4)(3, 7)(5, 10)(8, 12) \\
(1, 6, 7,11, 4, 2, 3, 9)(5, 12, 10, 8) \\
(1, 3, 4, 7)(2, 5, 11, 8, 6, 10, 9, 12)
\end{tabular}
}
\vspace{2mm}

The automorphism group of the
graph has order 384.  So the stabilizer of a vertex is a Klein $4$-group.
For a given vertex $v$, any neighbour of $v$ by way of a noninvolutory edge
can be mapped to any other such neighbour by an automorphism that fixes $v$.

In order to find smaller examples, we expanded the search to
include voltage graphs.  Some smaller graphs were obtained, and we noticed
that, unlike the example above, in each case the automorphism group of the graph
had a trivial vertex stabilizer.  As a result, we decided to focus the search on
exactly those graphs, i.e.,
graphs that have a semiregular automorphism group\footnote{Recall that a permutation is {\em semiregular} if all
of its cycles have the same length.} and whose vertex orbits have lengths approximately
$n/g$, where $n$ is the order and $g$ the girth.
Such graphs have been a subject of interest due to the {\em polycirculant
conjecture}~\cite{maruvsivc1981vertex}, which asserts that every vertex-transitive
digraph has a semiregular automorphism (see~\cite{semireg2014} for a nice
summary of progress on this topic).

Cubic graphs with semiregular automorphisms have been studied before, and
called {\em LCF graphs}, because they
were originally considered by Lederburg, Coxeter and Frucht~\cite{zerosymmetric}.
Their construction pertains to cubic graphs, but the idea is easily generalised.
So for convenience, and succinctness, we refer to a graph of composite order $n = r s$
that has a semiregular
automorphism composed of $r$ cycles of length $s$ as an $LCF(r,s)$ graph.
We label the vertices of such a graph as
\[
\{v_{i + r j} \ | \ 0 \leq i < r \ and\ 0 \leq j < s\}.
\]
\noindent
Thus the vertex orbits under the action of the group generated by the semiregular automorphism
are of the form

\[
\{v_{i + r j} \ |\  0 \leq j < s \}, \mbox{ for } 0 \leq i < r. 
\]

\noindent
The sets of potential edges are then partitioned into orbits of the form
\[
\{(v_{i + r j},v_{i + r j + t}) \ |\  0 \leq j < s \}
\]
for $ 1 \leq t \leq  n/2$.  All subscript addition is done modulo $rs$.

\subsubsection{Description of the LCF search method}
\label{sect:lcf_description}


We now outline our LCF search method which we will use to establish new upper bounds for the order of the smallest $4$-chromatic graphs of girth $g \geq 6$. (Note that this method also works for graphs with girth less than 6, cf.\ our new LCF triangle-free 6-chromatic graph on 40 vertices from Section~\ref{subsect:UB_tf}).

The biggest obstacle to a successful search is the fact that we ultimately
must compute the chromatic number, an NP-hard problem, of any candidate
graphs we find.
Consider the $4$-chromatic, girth $7$, case.
Here we were searching through graphs whose
orders are approximately $200$.
Searching through LCF graphs of these orders requires considering millions of graphs ({\em very} conservatively).
Determining whether or not one of these graphs has a $3$-colouring may
take several seconds.  Hence it is not feasible to precisely determine the
chromatic number of every graph we consider.
A fast  approximate colouring procedure was needed.
The procedure is a modified version
of the procedure used in the context of $4$-chromatic triangle-free unit-distance
graphs by the first author~\cite{unitdisttrifree}.  This procedure almost always predicts
the chromatic number correctly.   For graphs with orders $100$ to $300$
we know of $5$ cases (out of perhaps billions) where the procedure was wrong.
In each case the correct answer was determined by running the procedure twice.

Two versions of the main search program were designed: one to do complete searches
for $LCF(r,s)$ graphs, for given $r$ and $s$, and one which uses heuristics
to handle larger cases.  We describe the latter, more successful, version.
The first version of the procedure uses three external functions.

\begin{description}
\item [] $ randomColourable(k,G) $: The randomised colouring function that attempts
to colour the graph $G$ using $k$ colours.  Returns {\tt true} if a $k$-colouring of $G$ was found and {\tt false} if no $k$-colouring was found.
\item [] $ containsSmallCycles(g,G) $: Checks whether the graph $G$ contains any cycles whose length is {\em less than} $g$. Returns {\tt true} if the graph contains such a cycle, else returns {\tt false}.
\item [] $ getOrbits(r,s)$: A function that finds all possible semiregular
orbits for an $LCF(r,s)$ graph with a labelling as given earlier in Section~\ref{subsect:UB_higher_girth}.
\end{description}

The goal of the procedure is to find an $LCF(r,s)$ graph of girth at least $g$ that the
$randomColourable$ function fails to colour.  Such a graph is then a candidate to be
checked with a program that does an exhaustive search for colourings.
The general structure of the procedure is given in Algorithm~\ref{lcfsearch1}. Here $E(G)$ denotes
the edge set of $G$.

\begin{algorithm}[htb!]
\caption{Basic LCF Search}\label{lcfsearch1}
\begin{algorithmic}[1]
\Procedure{BasicSearch}{girth $g$, $r$, $s$}
\State $olist \gets getOrbits(r,s)$
\While {true}
\State Shuffle $olist$
\State $E(G) \gets \varnothing$
\For {$orb \in olist$}
\State $E(G) \gets E(G) \cup orb$
\If {$containsSmallCycles(g,G)$}
\State $E(G) \gets E(G) - orb$
\EndIf
\EndFor
\If {$\textit{not} \; randomColourable(3,G)$}
\State \Return{$G$}
\EndIf
\EndWhile
\EndProcedure
\end{algorithmic}
\end{algorithm}

This procedure is not capable of producing the graphs from Section~\ref{sect:results_LCF} which establish new upper bounds  without some refinement.  We will consider the case of even girth,
which is the more difficult case.
Intuitively the difficulty arises because to increase the chromatic number
of a graph, one needs to add a lot of edges;
but for even girth, the most effective way to add a lot of edges
is to create a bipartite graph.
Somehow our procedure must avoid the tendency to produce bipartite graphs.
One way to accomplish this is to attempt to maximise the number of odd cycles.
Counting all odd cycles is a prohibitively expensive computation,
so we focus on $g+1$-cycles.  The second procedure uses three new functions.

\begin{description}
\item [] $ bestOrbits(olist,G) $: Returns a list of the orbits
that can be added to the graph without creating any short cycles, but which
create the maximum number of new $g+1$ cycles.
\item [] $ updateOrbits(oldOrbitList,newOrbit,G) $: Returns a list of the orbits
in oldOrbitList than can be added to $G$ without creating any short cycles.  Since
orbits are added to the graph one at a time, knowledge of the most recently
added orbit is useful for efficiency.
\item [] $ randomChoice(list) $: Returns a random element of $list$.
\end{description}

The modified version of the procedure is given in Algorithm~\ref{lcfsearch2}.

\begin{algorithm}[htb!]
\caption{Even Girth LCF Search}\label{lcfsearch2}
\begin{algorithmic}[1]
\Procedure{EvenGirthSearch}{girth $g$, $r$, $s$}
\State $olist \gets getOrbits(r,s)$
\While {true}
\State $tmplist \gets olist$
\State $E(G) \gets \varnothing$
\While {$tmplist \neq \varnothing$}
\State $bestlist \gets bestOrbits(tmplist,G)$
\State $orb \gets randomChoice(bestlist)$
\State $E(G) \gets E(G) \cup orb$
\State $tmplist \gets updateOrbits(tmplist,orb,G)$
\EndWhile
\If {$not \; randomColourable(3,G)$}
\State \Return{$G$}
\EndIf
\EndWhile
\EndProcedure
\end{algorithmic}
\end{algorithm}

This gives the general idea of the search method.  In the interest of efficiency
a couple of heuristics were added.  First, instead of always using the $bestOrbits$
function in the inner {\tt while} loop in Algorithm~\ref{lcfsearch2}, some fraction of the time a random element
was chosen from $tmplist$.  This avoids calls to $bestOrbits$ where most of the
processor time is spent.  It also mitigates against any tendency for the outer
{\tt while} loop to repeatedly check the same graph.

A second modification is to require that chromatic number $3$ is
achieved early in the process.  So after a specified number of edges have been
added in the inner {\tt while} loop, we check whether any odd cycles have yet
appeared in the graph.  If not, we break out of the loop and restart the outer
loop.


\subsubsection{Results obtained using the LCF search method}
\label{sect:results_LCF}

The first new result we obtained using the LCF search method
is an order 66 $LCF(6,11)$ graph with chromatic number $4$
and girth $6$, significantly smaller than our Cayley graph of order 96 which we described at the beginning of Section~\ref{subsect:UB_higher_girth}.
The graph is given in Table~\ref{table:ub_n64}.  The table of an $LCF(r,s)$ graph should be
interpreted as follows.
The rows of the table are
labelled from $0$ to $r-1$.  An entry of $t$ in row $i$ indicates an orbit of
the type specified at the beginning of Section~\ref{subsect:UB_higher_girth}: the
graph contains the edges $(v_{i + r j}, v_{i + r j + t})$ for $0 \leq j < s$.
(Recall that the subscript addition $i + rj + t$ is done modulo $rs$, so it also makes sense when $t$ is negative).
Some of the entries in the table are redundant, for example, the $1$ entry in
row $0$ determines the same set of edges as the $-1$ entry in row~$1$.
However the redundancy makes it clear that the
graph is $5$-regular.

\begin{table}[ht!b]
\centering
\begin{tabular}{r r r r r r}
0:  &  1  &  6  &  $-23$  &  $-6$  &  $-1$  \\
1:  &  1  &  9  &  14  &  23  &  $-1$  \\
2:  &  1  &  26  &  33  &  $-10$  &  $-1$  \\
3:  &  1  &  18  &  $-18$  &  $-14$  &  $-1$  \\
4:  &  1  &  10  &  $-26$  &  $-9$  &  $-1$  \\
5:  &  1  &  18  &  33  &  $-18$  &  $-1$  \\
\end{tabular}

\caption{An $LCF(6,11)$ graph on 66 vertices with chromatic number $4$ and girth $6$ listed in LCF format.}

\label{table:ub_n64}
\end{table}

This graph has 66 vertices and is small enough that its chromatic number can be  verified
using any of the standard symbolic Mathematics software packages
(Sage, Maple, Mathematica).  We believe that this is the smallest known
$4$-chromatic graph of girth $6$ and thus yields the following upper bound for $n_6(4)$.


\begin{theorem}\label{thm:ub_n64}
$n_6(4) \leq 66$.
\end{theorem}

For comparison purpose we note that the smallest $4$-chromatic graph of girth $6$ obtained by Descartes'
construction~\cite{descartes1954} has $352\,735$ vertices. 

The next result deals with $4$-chromatic graphs of girth 7.  The construction we
present has 171 vertices and is an $LCF(9,19)$ graph and is listed in Table~\ref{table:ub_n74}.

\begin{table}[ht!b]
\centering
\begin{tabular}{r r r r r r}
0:  &  1  &  72  &  $-72$  &  $-13$  &  $-1$  \\
1:  &  1  &  77  &  $-68$  &  $-34$  &  $-1$  \\
2:  &  1  &  14  &  67  &  $-85$  &  $-1$  \\
3:  &  1  &  23  &  34  &  55  &  $-1$  \\
4:  &  1  &  38  &  $-55$  &  $-8$  &  $-1$  \\
5:  &  1  &  8  &  13  &  68  &  $-1$  \\
6:  &  1  &  $-77$  &  $-67$  &  $-38$  &  $-1$  \\
7:  &  1  &  46  &  85  &  $-14$  &  $-1$  \\
8:  &  1  &  $-46$  &  $-23$  &  $-1$  &    \\
\end{tabular}

\caption{An $LCF(9,19)$ graph on 171 vertices with chromatic number $4$ and girth $7$ listed in LCF format.}

\label{table:ub_n74}
\end{table}

Using Sage, it takes approximately one hour to verify the chromatic number of this graph.
This leads to the following new bound for $n_7(4)$.


\begin{theorem}\label{thm:ub_n74}
$n_7(4) \leq 171$.
\end{theorem}

The graphs from Table~\ref{table:ub_n64} and~\ref{table:ub_n74} can also
be downloaded from the database of interesting graphs from the \textit{House of
Graphs}~\cite{hog} by searching for the keywords ``$4$-chromatic girth 6'' and
``$4$-chromatic girth 7''\footnote{These graphs can also be accessed directly at \url{https://hog.grinvin.org/ViewGraphInfo.action?id=30637} and \url{https://hog.grinvin.org/ViewGraphInfo.action?id=30639}},
respectively. We also verified that these graphs are vertex-critical.

The next case we consider is $n_5(5)$.  The smallest example we have been able to
find is a Cayley graph of order $80$,  first constructed by Royle (personal communication),
in the context of the
cage problem~\cite{exoo2008dynamic}.  It is the smallest known regular
graph of degree $8$ and girth~$5$. 
Our LCF search program was able to independently reproduce
this graph and to determine that it is $5$-chromatic and was unable to find any smaller examples. So we have $n_5(5) \leq 80$.
This graph is listed in LCF format in Table~\ref{table:ub_n55} and it can also be downloaded from the
\textit{House of Graphs}~\cite{hog} by searching for the keywords ``$5$-chromatic girth
5''\footnote{This graph can also be accessed directly at \url{https://hog.grinvin.org/ViewGraphInfo.action?id=30635}}.

\begin{table}[ht!b]
\begin{center}
\begin{tabular}{r r r r r r r r r r r r r r}
0: & 1 & 19 & 32 & -35 & -32 & -27 & -23 & -1 \\
1: & 1 & 16 & 23 & 27 & 35 & -19 & -16 & -1 \\ 
2: & 1 & 5 & 13 & 19 & 32 & -32 & -23 & -1 \\ 
3: & 1 & 16 & 23 & -19 & -16 & -13 & -5 & -1 \\
\end{tabular}
\end{center}
\caption{An $LCF(4,20)$ graph on 80 vertices with chromatic number $5$ and girth $5$ listed in LCF format.}

\label{table:ub_n55}
\end{table}

In addition to the graphs reported above, several good candidates for other cases
were found.  Unfortunately these graphs seem too large to have their chromatic number
precisely determined in a reasonable amount of time. One of these graphs is listed in Table~\ref{table:lcf_355v}: a graph of girth 5 on 355 vertices which we suspect to be $6$-chromatic.

\begin{table}[ht!b]

\begin{center}
\begin{tabular}{r r r r r r r r r r r r r r}
0: & 1 & 24 & 45 & 61 & 101 & 128 & -148 & -82 & -79 & -69 & -64 & -45 & -1 \\
1: & 1 & 64 & 69 & 79 & 96 & 155 & 177 & -155 & -123 & -101 & -61 & -7 & -1 \\
2: & 1 & 17 & 27 & 36 & 47 & 51 & 90 & 148 & -168 & -108 & -96 & -90 & -1 \\
3: & 1 & 41 & 70 & 82 & 123 & 131 & -177 & -128 & -70 & -51 & -36 & -1 &  \\
4: & 1 & 7 & 108 & 168 & 175 & -175 & -131 & -47 & -41 & -27 & -24 & -17 & -1 \\
\end{tabular}
\end{center}

\caption{An $LCF(5,71)$ graph of girth 5 on 355 vertices listed in LCF format which is possibly 6-chromatic.}

\label{table:lcf_355v}
\end{table}

\section{Open problems}
\label{sect:open_problems}

We conclude with the following open problems. 

\begin{question}
Does every smallest $k$-chromatic graph of girth at least $g$ have girth equal to~$g$?
\end{question}

The analogous question for cages (smallest regular graphs of given degree and girth) was
answered positively by Erd\H{o}s and Sachs \cite{erdossachs}.
They showed that for degree $d \ge 3$ and girth $g \ge 3$, 
a smallest regular graph of degree $d$ and girth at least $g$ has girth exactly $g$.

\begin{question}
Is there a construction from which it follows that $n_g(k + 1) \leq c \cdot n_g(k)$ for a constant $c$ and $g \geq 5$?
\end{question}

Recall that for $g=4$ it follows from the Mycielski construction that $n_4(k+1) \leq 2n_4(k) + 1$.


\acknowledgements
\label{sec:ack}
We would like to thank Staszek Radziszowski as well as the anonymous referees for useful suggestions.
Most of the computations were carried out using the Stevin Supercomputer Infrastructure at Ghent University and the computers at the Computer Science labs
at Indiana State University.


\bibliographystyle{plain}
\bibliography{references}
\label{sec:biblio}


\section*{Appendix}

Note: the adjacency lists of the graphs which establish a new upper bound for $n_g(k)$ can also be found in the ``ancillary material" of this paper on arXiv.

\subsection*{A triangle-free 6-chromatic graph on 40 vertices}

Below is one of the triangle-free 6-chromatic graphs on 40 vertices. It is an $LCF(8,5)$ graph and has an automorphism group of size 10. This graph is listed in \textit{LCF format} to keep things concise. Please refer to Section~\ref{subsect:UB_higher_girth} for the definition of this format.

\begin{center}
\begin{tabular}{r r r r r r r r r r r r r r}
0:  &  1  &  5  &  14  &  16  &  $-18$  &  $-16$  &  $-12$  &  $-9$  &  $-7$  &  $-4$  &  $-1$  &    &   \\
1:  &  1  &  7  &  9  &  17  &  20  &  $-15$  &  $-12$  &  $-7$  &  $-1$  &    &    &    &   \\
2:  &  1  &  5  &  7  &  10  &  13  &  15  &  18  &  20  &  $-17$  &  $-12$  &  $-9$  &  $-7$  &  $-1$ \\
3:  &  1  &  3  &  7  &  16  &  18  &  $-16$  &  $-14$  &  $-12$  &  $-4$  &  $-1$  &    &    &   \\
4:  &  1  &  4  &  12  &  16  &  19  &  $-18$  &  $-16$  &  $-10$  &  $-7$  &  $-1$  &    &    &   \\
5:  &  1  &  7  &  12  &  14  &  16  &  20  &  $-18$  &  $-16$  &  $-7$  &  $-5$  &  $-1$  &    &   \\
6:  &  1  &  7  &  12  &  16  &  18  &  20  &  $-16$  &  $-14$  &  $-3$  &  $-1$  &    &    &   \\
7:  &  1  &  4  &  9  &  12  &  16  &  $-19$  &  $-16$  &  $-13$  &  $-5$  &  $-1$  &    &    &   \\
\end{tabular}
\end{center}

\subsection*{A triangle-free 7-chromatic graph on 77 vertices}

Below is the adjacency list of one of the triangle-free 7-chromatic graphs on 77 vertices which yields the upper bound from Theorem~\ref{thm:ub_n47}. It has an automorphism group of size 10.

\begin{center}
\setlength{\tabcolsep}{2pt} 
\renewcommand{\arraystretch}{0.9}
\small
\begin{longtable}{r r r r r r r r r r r r r r r r r r r r r r r r r r r}
0: & 25 & 29 & 31 & 33 & 35 & 36 & 37 & 38 & 39 & 60 & 64 & 66 & 68 & 70 & 71 & 72 & 73 & 74 & 75 &  &  &  &  &  &  &  \\
1: & 28 & 29 & 30 & 31 & 35 & 36 & 37 & 38 & 39 & 63 & 64 & 65 & 66 & 70 & 71 & 72 & 73 & 74 & 75 &  &  &  &  &  &  &  \\
2: & 25 & 27 & 32 & 33 & 35 & 36 & 37 & 38 & 39 & 60 & 62 & 67 & 68 & 70 & 71 & 72 & 73 & 74 & 75 &  &  &  &  &  &  &  \\
3: & 26 & 28 & 30 & 34 & 35 & 36 & 37 & 38 & 39 & 61 & 63 & 65 & 69 & 70 & 71 & 72 & 73 & 74 & 75 &  &  &  &  &  &  &  \\
4: & 26 & 27 & 32 & 34 & 35 & 36 & 37 & 38 & 39 & 61 & 62 & 67 & 69 & 70 & 71 & 72 & 73 & 74 & 75 &  &  &  &  &  &  &  \\
5: & 8 & 9 & 10 & 24 & 26 & 28 & 32 & 34 & 37 & 39 & 43 & 44 & 45 & 59 & 61 & 63 & 67 & 69 & 72 & 74 &  &  &  &  &  &  \\
6: & 7 & 9 & 14 & 21 & 26 & 27 & 32 & 33 & 35 & 39 & 42 & 44 & 49 & 56 & 61 & 62 & 67 & 68 & 70 & 74 &  &  &  &  &  &  \\
7: & 6 & 8 & 13 & 23 & 28 & 29 & 30 & 34 & 37 & 38 & 41 & 43 & 48 & 58 & 63 & 64 & 65 & 69 & 72 & 73 &  &  &  &  &  &  \\
8: & 5 & 7 & 12 & 20 & 25 & 27 & 31 & 33 & 35 & 36 & 40 & 42 & 47 & 55 & 60 & 62 & 66 & 68 & 70 & 71 &  &  &  &  &  &  \\
9: & 5 & 6 & 11 & 22 & 25 & 29 & 30 & 31 & 36 & 38 & 40 & 41 & 46 & 57 & 60 & 64 & 65 & 66 & 71 & 73 &  &  &  &  &  &  \\
10: & 5 & 11 & 12 & 15 & 20 & 21 & 30 & 31 & 35 & 36 & 40 & 46 & 47 & 50 & 55 & 56 & 65 & 66 & 70 & 71 &  &  &  &  &  &  \\
11: & 9 & 10 & 14 & 17 & 23 & 24 & 32 & 33 & 37 & 39 & 44 & 45 & 49 & 52 & 58 & 59 & 67 & 68 & 72 & 74 &  &  &  &  &  &  \\
12: & 8 & 10 & 13 & 16 & 22 & 23 & 32 & 34 & 37 & 38 & 43 & 45 & 48 & 51 & 57 & 58 & 67 & 69 & 72 & 73 &  &  &  &  &  &  \\
13: & 7 & 12 & 14 & 19 & 21 & 24 & 31 & 33 & 35 & 39 & 42 & 47 & 49 & 54 & 56 & 59 & 66 & 68 & 70 & 74 &  &  &  &  &  &  \\
14: & 6 & 11 & 13 & 18 & 20 & 22 & 30 & 34 & 36 & 38 & 41 & 46 & 48 & 53 & 55 & 57 & 65 & 69 & 71 & 73 &  &  &  &  &  &  \\
15: & 10 & 16 & 17 & 22 & 25 & 27 & 32 & 33 & 37 & 38 & 45 & 51 & 52 & 57 & 60 & 62 & 67 & 68 & 72 & 73 &  &  &  &  &  &  \\
16: & 12 & 15 & 19 & 24 & 28 & 29 & 30 & 31 & 35 & 39 & 47 & 50 & 54 & 59 & 63 & 64 & 65 & 66 & 70 & 74 &  &  &  &  &  &  \\
17: & 11 & 15 & 18 & 21 & 26 & 28 & 30 & 34 & 35 & 36 & 46 & 50 & 53 & 56 & 61 & 63 & 65 & 69 & 70 & 71 &  &  &  &  &  &  \\
18: & 14 & 17 & 19 & 23 & 25 & 29 & 31 & 33 & 37 & 39 & 49 & 52 & 54 & 58 & 60 & 64 & 66 & 68 & 72 & 74 &  &  &  &  &  &  \\
19: & 13 & 16 & 18 & 20 & 26 & 27 & 32 & 34 & 36 & 38 & 48 & 51 & 53 & 55 & 61 & 62 & 67 & 69 & 71 & 73 &  &  &  &  &  &  \\
20: & 8 & 10 & 14 & 19 & 23 & 24 & 28 & 29 & 37 & 39 & 43 & 45 & 49 & 54 & 58 & 59 & 63 & 64 & 72 & 74 &  &  &  &  &  &  \\
21: & 6 & 10 & 13 & 17 & 22 & 23 & 25 & 29 & 37 & 38 & 41 & 45 & 48 & 52 & 57 & 58 & 60 & 64 & 72 & 73 &  &  &  &  &  &  \\
22: & 9 & 12 & 14 & 15 & 21 & 24 & 26 & 28 & 35 & 39 & 44 & 47 & 49 & 50 & 56 & 59 & 61 & 63 & 70 & 74 &  &  &  &  &  &  \\
23: & 7 & 11 & 12 & 18 & 20 & 21 & 26 & 27 & 35 & 36 & 42 & 46 & 47 & 53 & 55 & 56 & 61 & 62 & 70 & 71 &  &  &  &  &  &  \\
24: & 5 & 11 & 13 & 16 & 20 & 22 & 25 & 27 & 36 & 38 & 40 & 46 & 48 & 51 & 55 & 57 & 60 & 62 & 71 & 73 &  &  &  &  &  &  \\
25: & 0 & 2 & 8 & 9 & 15 & 18 & 21 & 24 & 26 & 28 & 34 & 43 & 44 & 50 & 53 & 56 & 59 & 61 & 63 & 69 &  &  &  &  &  &  \\
26: & 3 & 4 & 5 & 6 & 17 & 19 & 22 & 23 & 25 & 29 & 31 & 40 & 41 & 52 & 54 & 57 & 58 & 60 & 64 & 66 &  &  &  &  &  &  \\
27: & 2 & 4 & 6 & 8 & 15 & 19 & 23 & 24 & 28 & 29 & 30 & 41 & 43 & 50 & 54 & 58 & 59 & 63 & 64 & 65 &  &  &  &  &  &  \\
28: & 1 & 3 & 5 & 7 & 16 & 17 & 20 & 22 & 25 & 27 & 33 & 40 & 42 & 51 & 52 & 55 & 57 & 60 & 62 & 68 &  &  &  &  &  &  \\
29: & 0 & 1 & 7 & 9 & 16 & 18 & 20 & 21 & 26 & 27 & 32 & 42 & 44 & 51 & 53 & 55 & 56 & 61 & 62 & 67 &  &  &  &  &  &  \\
30: & 1 & 3 & 7 & 9 & 10 & 14 & 16 & 17 & 27 & 32 & 33 & 42 & 44 & 45 & 49 & 51 & 52 & 62 & 67 & 68 &  &  &  &  &  &  \\
31: & 0 & 1 & 8 & 9 & 10 & 13 & 16 & 18 & 26 & 32 & 34 & 43 & 44 & 45 & 48 & 51 & 53 & 61 & 67 & 69 &  &  &  &  &  &  \\
32: & 2 & 4 & 5 & 6 & 11 & 12 & 15 & 19 & 29 & 30 & 31 & 40 & 41 & 46 & 47 & 50 & 54 & 64 & 65 & 66 &  &  &  &  &  &  \\
33: & 0 & 2 & 6 & 8 & 11 & 13 & 15 & 18 & 28 & 30 & 34 & 41 & 43 & 46 & 48 & 50 & 53 & 63 & 65 & 69 &  &  &  &  &  &  \\
34: & 3 & 4 & 5 & 7 & 12 & 14 & 17 & 19 & 25 & 31 & 33 & 40 & 42 & 47 & 49 & 52 & 54 & 60 & 66 & 68 &  &  &  &  &  &  \\
35: & 0 & 1 & 2 & 3 & 4 & 6 & 8 & 10 & 13 & 16 & 17 & 22 & 23 & 41 & 43 & 45 & 48 & 51 & 52 & 57 & 58 &  &  &  &  &  \\
36: & 0 & 1 & 2 & 3 & 4 & 8 & 9 & 10 & 14 & 17 & 19 & 23 & 24 & 43 & 44 & 45 & 49 & 52 & 54 & 58 & 59 &  &  &  &  &  \\
37: & 0 & 1 & 2 & 3 & 4 & 5 & 7 & 11 & 12 & 15 & 18 & 20 & 21 & 40 & 42 & 46 & 47 & 50 & 53 & 55 & 56 &  &  &  &  &  \\
38: & 0 & 1 & 2 & 3 & 4 & 7 & 9 & 12 & 14 & 15 & 19 & 21 & 24 & 42 & 44 & 47 & 49 & 50 & 54 & 56 & 59 &  &  &  &  &  \\
39: & 0 & 1 & 2 & 3 & 4 & 5 & 6 & 11 & 13 & 16 & 18 & 20 & 22 & 40 & 41 & 46 & 48 & 51 & 53 & 55 & 57 &  &  &  &  &  \\
40: & 8 & 9 & 10 & 24 & 26 & 28 & 32 & 34 & 37 & 39 & 75 & 76 &  &  &  &  &  &  &  &  &  &  &  &  &  &  \\
41: & 7 & 9 & 14 & 21 & 26 & 27 & 32 & 33 & 35 & 39 & 75 & 76 &  &  &  &  &  &  &  &  &  &  &  &  &  &  \\
42: & 6 & 8 & 13 & 23 & 28 & 29 & 30 & 34 & 37 & 38 & 75 & 76 &  &  &  &  &  &  &  &  &  &  &  &  &  &  \\
43: & 5 & 7 & 12 & 20 & 25 & 27 & 31 & 33 & 35 & 36 & 75 & 76 &  &  &  &  &  &  &  &  &  &  &  &  &  &  \\
44: & 5 & 6 & 11 & 22 & 25 & 29 & 30 & 31 & 36 & 38 & 75 & 76 &  &  &  &  &  &  &  &  &  &  &  &  &  &  \\
45: & 5 & 11 & 12 & 15 & 20 & 21 & 30 & 31 & 35 & 36 & 75 & 76 &  &  &  &  &  &  &  &  &  &  &  &  &  &  \\
46: & 9 & 10 & 14 & 17 & 23 & 24 & 32 & 33 & 37 & 39 & 75 & 76 &  &  &  &  &  &  &  &  &  &  &  &  &  &  \\
47: & 8 & 10 & 13 & 16 & 22 & 23 & 32 & 34 & 37 & 38 & 75 & 76 &  &  &  &  &  &  &  &  &  &  &  &  &  &  \\
48: & 7 & 12 & 14 & 19 & 21 & 24 & 31 & 33 & 35 & 39 & 75 & 76 &  &  &  &  &  &  &  &  &  &  &  &  &  &  \\
49: & 6 & 11 & 13 & 18 & 20 & 22 & 30 & 34 & 36 & 38 & 75 & 76 &  &  &  &  &  &  &  &  &  &  &  &  &  &  \\
50: & 10 & 16 & 17 & 22 & 25 & 27 & 32 & 33 & 37 & 38 & 75 & 76 &  &  &  &  &  &  &  &  &  &  &  &  &  &  \\
51: & 12 & 15 & 19 & 24 & 28 & 29 & 30 & 31 & 35 & 39 & 75 & 76 &  &  &  &  &  &  &  &  &  &  &  &  &  &  \\
52: & 11 & 15 & 18 & 21 & 26 & 28 & 30 & 34 & 35 & 36 & 75 & 76 &  &  &  &  &  &  &  &  &  &  &  &  &  &  \\
53: & 14 & 17 & 19 & 23 & 25 & 29 & 31 & 33 & 37 & 39 & 75 & 76 &  &  &  &  &  &  &  &  &  &  &  &  &  &  \\
54: & 13 & 16 & 18 & 20 & 26 & 27 & 32 & 34 & 36 & 38 & 75 & 76 &  &  &  &  &  &  &  &  &  &  &  &  &  &  \\
55: & 8 & 10 & 14 & 19 & 23 & 24 & 28 & 29 & 37 & 39 & 75 & 76 &  &  &  &  &  &  &  &  &  &  &  &  &  &  \\
56: & 6 & 10 & 13 & 17 & 22 & 23 & 25 & 29 & 37 & 38 & 75 & 76 &  &  &  &  &  &  &  &  &  &  &  &  &  &  \\
57: & 9 & 12 & 14 & 15 & 21 & 24 & 26 & 28 & 35 & 39 & 75 & 76 &  &  &  &  &  &  &  &  &  &  &  &  &  &  \\
58: & 7 & 11 & 12 & 18 & 20 & 21 & 26 & 27 & 35 & 36 & 75 & 76 &  &  &  &  &  &  &  &  &  &  &  &  &  &  \\
59: & 5 & 11 & 13 & 16 & 20 & 22 & 25 & 27 & 36 & 38 & 75 & 76 &  &  &  &  &  &  &  &  &  &  &  &  &  &  \\
60: & 0 & 2 & 8 & 9 & 15 & 18 & 21 & 24 & 26 & 28 & 34 & 76 &  &  &  &  &  &  &  &  &  &  &  &  &  &  \\
61: & 3 & 4 & 5 & 6 & 17 & 19 & 22 & 23 & 25 & 29 & 31 & 76 &  &  &  &  &  &  &  &  &  &  &  &  &  &  \\
62: & 2 & 4 & 6 & 8 & 15 & 19 & 23 & 24 & 28 & 29 & 30 & 76 &  &  &  &  &  &  &  &  &  &  &  &  &  &  \\
63: & 1 & 3 & 5 & 7 & 16 & 17 & 20 & 22 & 25 & 27 & 33 & 76 &  &  &  &  &  &  &  &  &  &  &  &  &  &  \\
64: & 0 & 1 & 7 & 9 & 16 & 18 & 20 & 21 & 26 & 27 & 32 & 76 &  &  &  &  &  &  &  &  &  &  &  &  &  &  \\
65: & 1 & 3 & 7 & 9 & 10 & 14 & 16 & 17 & 27 & 32 & 33 & 76 &  &  &  &  &  &  &  &  &  &  &  &  &  &  \\
66: & 0 & 1 & 8 & 9 & 10 & 13 & 16 & 18 & 26 & 32 & 34 & 76 &  &  &  &  &  &  &  &  &  &  &  &  &  &  \\
67: & 2 & 4 & 5 & 6 & 11 & 12 & 15 & 19 & 29 & 30 & 31 & 76 &  &  &  &  &  &  &  &  &  &  &  &  &  &  \\
68: & 0 & 2 & 6 & 8 & 11 & 13 & 15 & 18 & 28 & 30 & 34 & 76 &  &  &  &  &  &  &  &  &  &  &  &  &  &  \\
69: & 3 & 4 & 5 & 7 & 12 & 14 & 17 & 19 & 25 & 31 & 33 & 76 &  &  &  &  &  &  &  &  &  &  &  &  &  &  \\
70: & 0 & 1 & 2 & 3 & 4 & 6 & 8 & 10 & 13 & 16 & 17 & 22 & 23 & 76 &  &  &  &  &  &  &  &  &  &  &  &  \\
71: & 0 & 1 & 2 & 3 & 4 & 8 & 9 & 10 & 14 & 17 & 19 & 23 & 24 & 76 &  &  &  &  &  &  &  &  &  &  &  &  \\
72: & 0 & 1 & 2 & 3 & 4 & 5 & 7 & 11 & 12 & 15 & 18 & 20 & 21 & 76 &  &  &  &  &  &  &  &  &  &  &  &  \\
73: & 0 & 1 & 2 & 3 & 4 & 7 & 9 & 12 & 14 & 15 & 19 & 21 & 24 & 76 &  &  &  &  &  &  &  &  &  &  &  &  \\
74: & 0 & 1 & 2 & 3 & 4 & 5 & 6 & 11 & 13 & 16 & 18 & 20 & 22 & 76 &  &  &  &  &  &  &  &  &  &  &  &  \\
75: & 0 & 1 & 2 & 3 & 4 & 40 & 41 & 42 & 43 & 44 & 45 & 46 & 47 & 48 & 49 & 50 & 51 & 52 & 53 & 54 & 55 & 56 & 57 & 58 & 59 &  \\
76: & 40 & 41 & 42 & 43 & 44 & 45 & 46 & 47 & 48 & 49 & 50 & 51 & 52 & 53 & 54 & 55 & 56 & 57 & 58 & 59 & 60 & 61 & 62 & 63 & 64 & 65 \\
 & 66 & 67 & 68 & 69 & 70 & 71 & 72 & 73 & 74 &  &  &  &  &  &  &  &  &  &  &  &  &  &  &  &  &  \\
\end{longtable}
\end{center}

\end{document}